\documentclass[10pt, a4paper]{amsart}
\usepackage{amssymb, amsmath, amsfonts, verbatim}
\usepackage{color}
\newtheorem{lemma}[equation]{Lemma}
\newtheorem{theorem}[equation]{Theorem}
\newenvironment{proofofT}[1]{\normalsize {\sc Proof of #1}:}{{\hfill $\Box$}}
\newtheorem{definition}[equation]{Definition}
\newtheorem{cor}[equation]{Corollary}
\theoremstyle{remark}
\newtheorem{remark}[equation]{Remark}
\newtheorem{question}[equation]{Question}
\newtheorem{example}[equation]{Example}
\newtheorem{conjecture}[equation]{Conjecture}

\title{A Note on Beauville $p$-Groups}
\date{8th November 2011. Updated version: 14th February 2012}

\author{Nathan Barker}
\address{Nathan Barker, School of Mathematics and Statistics, Newcastle University, Newcastle upon Tyne, NE1 7RU, United Kingdom}
\email{nathan.barker@ncl.ac.uk}

\author{Nigel Boston}
\address{Nigel Boston, Department of Mathematics, 303 Van Vleck Hall, 480 Lincoln Drive, Madison, WI 53706, USA}
\email{boston@math.wisc.edu}

\author{Ben Fairbairn}
\address{Ben Fairbairn, Department of Economics, Mathematics and Statistics, Birkbeck, University of London, Malet Street, London WC1E 7HX, United Kingdom}
\email{bfairbairn@ems.bbk.ac.uk}

\begin{document}
\maketitle

\begin{abstract}
We examine which $p$-groups of order $\le p^6$ are Beauville. We completely classify them for groups of order $\le p^4$. We also show that the proportion of $2$-generated groups of order $p^5$ which are Beauville tends to $1$ as $p$ tends to infinity; this is not true, however, for groups of order $p^6$. For each prime $p$ we determine the smallest non-abelian Beauville $p$-group.
\end{abstract}

\section{Introduction}

Let $G$ be a finite group. We call $G$ a \emph{Beauville group} if there exists a `Beauville structure' for $G$, which we define as follows.

\begin{definition}\label{maindef}
Let $G$ be a finite group. Let $x,y\in G$ and
$$\Sigma(x,y):=\bigcup_{i=1}^{|G|}\bigcup_{g\in G}\{(x^i)^g,(y^i)^g,((xy)^i)^g\}.$$

A \emph{Beauville structure} for $G$ is a pair of generating sets $\langle x_1,y_1\rangle=\langle x_2,y_2\rangle=G$ such that
$$\Sigma(x_1,y_1)\cap\Sigma(x_2,y_2)=\{e\}.$$
\end{definition}

Traditionally, authors have stated the above structure in terms of \emph{spherical systems of generators} of length $3$, meaning  $\{x,y,z\}$ with $xyz=e$, but we omit $z=(xy)^{-1}$ from our notation in this note. The structure above is often called an \emph{unmixed} Beauville structure; we do not, however, consider the mixed structures here. Furthermore, many earlier papers on Beauville structures add the condition that for $i=1,2$ we have $o(x_i)^{-1}+o(y_i)^{-1}+o(x_iy_i)^{-1}<1$, but this condition was subsequently found to be unnecessary \cite{BCG}.

Beauville groups were originally introduced in connection with a class of complex surfaces of general type, known as Beauville surfaces. These surfaces possess many useful geometric properties; their automorphism groups \cite{Jonesauts} and fundamental groups \cite{Catnese} are relatively easy to compute and are rigid surfaces in the sense of admitting no non-trivial deformations \cite{BCG2} and thus correspond to isolated points in the moduli space of surfaces of general type. 

In \cite[Question 7.7]{BCG2} Bauer, Catanese and Grunewald asked which groups are Beauville groups. In \cite{Catnese} Catanese classified the abelian Beauville groups by proving the following. We write $C_n$ for the cyclic group of order $n$.

\begin{theorem}[Catanese 2000]\label{Cat}
Let $G$ be an abelian Beauville group. Then $G=C_n\times C_n$ where $gcd(n,6)=1$.
\end{theorem}

After abelian groups, the next most natural class to consider are the nilpotent groups. The following (and its converse) is an easy exercise for the reader.

\begin{lemma}
Let $G$ and $G'$ be Beauville groups and let $\{\{x_1,y_1\},\{x_2,y_2\}\}$ and $\{\{x'_1,y'_1\},\{x'_2,y'_2\}\}$ be their respective Beauville structures. Suppose that for $i=1,2$
$$gcd(o(x_i),o(x_i'))=gcd(o(y_i),o(y_i'))=1.$$
Then $\{\{(x_1,x'_1),(y_1,y'_1)\},\{(x_2,x'_2),(y_2,y'_2)\}\}$ is a Beauville structure for the group $G\times G'$.
\end{lemma}

Recalling that a finite group is nilpotent if and only if it is a direct product of its Sylow subgroups, the above lemma reduces the study of nilpotent Beauville groups to the study of Beauville $p$-groups, which is the case we focus on here. Notice that Theorem \ref{Cat} gives us an infinite supply of Beauville $p$-groups for every $p\geq5$ - simply let $n$ be a power of $p$. Various examples of non-abelian Beauville $p$-groups for specific values of $p$ have appeared elsewhere in the literature \cite{Barker,BarkerOther,BCG3,FGJ}, but little has been said about the general case.

In several places we shall refer to computer calculations that can easily be performed in Magma \cite{magma} or GAP \cite{GAP4}. In particular we will find it convenient to use the \texttt{SmallGroup(m,n)} notation that denotes the $n^{th}$ group of order $m$ that can be found in the small groups library of MAGMA or GAP \cite{library}. 

In addition, for each group presentation $\langle X|R\rangle$, if $a,b\in X$ commute, the relation $[a,b]=e$ will be omitted for economy of space.

We now summarize the main results of this paper. In Section \ref{General} we show that there exists a non-abelian Beauville group for each order $p^n$, $p\ge 5$ $n\ge 4$.  Sections \ref{OrderP3} and \ref{OrderP4} classify the non-abelian Beauville $p$-groups of order $p^3$ and $p^4$. 

In the penultimate section, we examine the groups of order $p^5$ and prove the following theorem.

\begin{theorem}\label{familyP5}
If $p>3$, then there exist at least $p+8$ Beauville groups of order $p^5$.
\end{theorem}

From the analysis of the number of $2$-generated groups of order $p^5$ we find the following consequence of the above theorem.

\begin{cor}\label{proportionP5}
The proportion of $2$-generated groups of order $p^5$ which are Beauville tends to $1$ as $p$ tends to infinity.
\end{cor}

For groups of order $p^6$ we find the following.

\begin{theorem}\label{familyP6}
If $p>3$, then there exist at least $p-1$ $2$-generated non-Beauville groups of order $p^6$.
\end{theorem}

From the analysis of the number of $2$-generated groups of order $p^6$ we find the following consequence of the above theorem.

\begin{cor}\label{proportionP6}
The proportion of $2$-generated groups of order $p^6$ which are Beauville does not tend to $1$ as $p$ tends to infinity.
\end{cor}

From \cite{FGJ} we have the following statement \textquotedblleft it is very plausible that most $2$-generated finite $p$-groups of sufficiently large order [are Beauville groups]". If we interpret that the word \textquotedblleft most" from the statement to mean that the proportion of Beauville groups tends to $1$ as $p$ tends to infinity, then this statement is true for groups of order $p^5$ but not for groups of order $p^6$.

\begin{question}
If $n>6$, what is the behavior, as $p$ tends to infinity, of the proportion of $2$-generated groups which are Beauville? 
\end{question}

Finally, through computational experimentation, we have the corollary of the combined results of this note.

\begin{cor}
The smallest non-abelian Beauville $p$-groups are
\begin{enumerate}
\item for $p=2$, {\textnormal{\ttfamily SmallGroup}}($2^7$,36);
\item for $p=3$, the group given by Example \ref{p3}, of order $3^5$;
\item for $p=5$, {\textnormal{\ttfamily SmallGroup}}($5^3$,$3$);
\item for $p\ge 7$, the groups given by Lemma \ref{NathanLem}, of order $p^3$.
\end{enumerate}
\end{cor}

\section{Some general results}\label{General}
We first explicitly show that there is a non-abelian $2$-generated non-Beauville group of order $p^n$ for every $n\geq3$ and for every prime $p$.

\begin{lemma}\label{firstgen}
The group
$$G:=\langle x,y|x^{p^n},y^p,x^y=x^{p^{n-1}+1}\rangle$$
is a non-abelian 2-generated non-Beauville group of order $p^{n+1}$ for every prime $p$ and every $n>1$. 
\end{lemma}

\begin{proof}
Clearly $G$ is non-abelian and 2-generated and a straightforward coset enumeration shows that the subgroup $\langle x\rangle$ has index $p$ and so $|G|=p^{n+1}$. Now, $Z(G)=\langle x^p\rangle$ and every element outside the subgroup $\langle x^p,y\rangle$ has order $p^n$. Consequently, any generating set must contain at least one element of order $p^n$, but all such elements power up to $x^{p^{n-1}}$ (i.e. there exists $a\in \mathbb{N}$ such that, for $w\in G$, $w^a=x^{p^{n-1}}$), so $G$ cannot have a Beauville structure.
\end{proof}

We remark that this lemma is a generalisation of \cite[Example 4A]{FJ} which is the case $n=2$. We now show that there exists a non-abelian Beauville group $G$ of order $p^n$ for every $p\ge 5$ and $n\ge 4$.

\begin{lemma}\label{BenLemma}
The group
$$G:=\langle x,y|x^{p^n},y^{p^n},x^y=x^{p+1}\rangle$$
is a non-abelian Beauville group of order $p^{2n}$ for every prime $p\ge5$ and every $n\ge 2$.
\end{lemma}

\begin{proof}
Clearly $G$ is non-abelian and 2-generated and a straightforward coset enumeration shows that the subgroup $\langle x\rangle$ has index $p^n$ and so $|G|=p^{2n}$. Let $p>5$ be prime. We claim that $\{\{x,y\},\{xy^2,xy^3\}\}$ is a Beauville structure in this case.

Now, every element of $G$ can be written as $x^iy^j$ for some $0\leq i,j\leq p^n-1$. Furthermore, since $Z(G)=\langle x^{p^{n-1}},y^{p^{n-1}}\rangle$ and so a necessary condition for two elements of $G$ to be conjugate is that they power up to the same elements of $Z(G)$. A straightforward induction tells us that 
$$(xy)^r=x^{1+(p+1)+(p+1)^2+\cdots+(p+1)^{r-1}}y^r.$$ 
An easy exercise in using geometric progressions and the binomial theorem tells us that for any prime $p$
$$1+(1+p)+\cdots+(1+p)^{p^{n-1}-1}\equiv p^{n-1}\mbox{ (mod }p^n).$$
Combining these two identities gives $(xy)^{p^{n-1}}=x^{p^{n-1}}y^{p^{n-1}}$. Similar identities can be established for the elements $xy^2$, $xy^3$ and $(xy^2xy^3)y^{-5}y^5=x^{1+(p+1)^2}y^5$, verifying that no powers of these elements are conjugate. 

Finally we need show these pairs generate. Clearly $\langle x,y\rangle=G$ by definition. Since $(xy^2)^{-1}xy^3=y$ and $xy^2y^{-2}=x$ so $G\leq\langle x,y\rangle\leq\langle xy^2,xy^3\rangle\leq G$.

Similar calculations in the case $p=5$ show that $\{\{x,y\},\{xy^2,xy^4\}\}$ is a Beauville structure.
\end{proof}

The above lemma has covered the groups of order an even power of a prime, $p^{2n}$. The next lemma covers the odd case, $p^{2n+1}$.

\begin{lemma}\label{Oddpowers}
The group
\[G:=\langle x,y,z,\alpha_{1},...,\alpha_{n-1},\beta_{1},...,\beta_{n-1}|x^{p^n},y^{p^n},z^p,[x,y]=z,\]
\[\alpha_i=x^{p^i},\beta_i=y^{p^i} \ \text{(for all $1 \le i\le n-1$)} \ \rangle,\] 
is a non-abelian Beauville group of order $p^{2n+1}$ for $p\ge 5$ and $n\ge 2$.
\end{lemma}

\begin{proof}
For $p\ge 5$ and $n\ge 2$, it is clear that $G$ is a $2$-generated group by $\{x,y\}$ and $\{xy^2,xy^4\}$. Furthermore, we have distinct subgroups $\langle x\rangle,\langle y\rangle, \langle z\rangle$ of $G$ of orders $p^n,p^n,p$ respectively. As every element of $G$ can be put in the form $x^iy^jz^k$, it follows that the order of $G$ is $p^{2n+1}$. 

We now claim the following is a Beauville structure $\{\{x,y\},\{xy^2,xy^4\}\}$ for $G$. Since $\alpha_{i},\beta_{i}\in Z(G)$ and $[x,y]=z$ we can construct the following $\Sigma$-sets,
\[\Sigma(x,y)=\{e\}\bigcup \left( \bigcup_{i=1}^{p^n-1}\{x^{i},y^{i},x^{i}y^{i}\}\langle z\rangle\right) \setminus \bigcup_{i=1}^{p^{n-1}-1}\bigcup_{j=1}^{p-1}\{x^{ip}z^j,y^{ip}z^j,x^{ip}y^{ip}z^j\},\]
and
\[\Sigma(xy^2,xy^4)=\]
\[\{e\}\bigcup \left(\bigcup_{i=1}^{p^n-1}\{x^{i}y^{2i},x^{i}y^{4i},x^{2i}y^{6i}\}\langle z\rangle\right)\setminus \bigcup_{i=1}^{p^{n-1}-1}\bigcup_{j=1}^{p-1}\{x^{ip}y^{2ip}z^j,x^{ip}y^{4ip}z^j,x^{2ip}y^{6ip}z^j\},\]
for this group. Here, we prefer to write the $\alpha_{i}$'s and $\beta_{j}$'s in terms of powers of $x^{p}$ and $y^p$ respectively. Therefore, $\Sigma(x,y)\cap\Sigma(xy^2,xy^4)=\{e\}$.
\end{proof}

\section{Groups of order $\leq p^3$}\label{OrderP3}

All groups of order $p$ or $p^2$ are abelian for every prime $p$. Thus, by Theorem \ref{Cat} the only Beauville group of order less than $p^3$ is $C_p\times C_p$ for $p>3$. There are no abelian Beauville groups of order $p^3$.

The classification of groups of order $p^3$ is well-known; this result is due to \cite{Holder}. There are two non-abelian groups of order $p^3$. The first is of the form discussed in Lemma \ref{firstgen} and is thus not a Beauville group. The second is taken care of by the following, which is a special case of Lemma \ref{Oddpowers}.

\begin{lemma}\label{NathanLem}
For any prime $p\ge 7$ the group
$$G:=\langle x,y,z|x^p,y^p,z^p,[x,y]=z\rangle$$
is a non-abelian Beauville group of order $p^3$ with Beauville structure $\{\{x,y\},\{xy^2,xy^3\}\}$.
\end{lemma}

\begin{proof}
The group $G$ is the extra-special plus type group $p_+^{1+2}$. Since $xyx^{-1}y^{-1}=[x,y]=z$ we have that $xyx^{-1}=yz$ and since $C_G(y^i)=\langle y,z\rangle$ for $1\leq i<p$ we see that the conjugates of $y^i$ are precisely the elements $y^iz^j$ for $1\leq j\leq p$. Similarly $C_{G}(g)=\langle g,z\rangle$ for all $g\in G\setminus Z(G)$.

Therefore, the condition, $\Sigma(x,y)\cap\Sigma(xy^2,xy^3)=\{e\}$ is equivalent to

$$(C_G(x)\cup C_G(y)\cup C_G(xy))\cap(C_G(xy^2)\cup C_G(xy^3)\cup C_G(xy^2xy^3))=Z(G)$$

Again, this can be shown to be equivalent to checking the equations $khk^{-1}\ne h$ for all $k\in \{x,y,(xy)^{-1}\}$ and $h\in \{xy^2,xy^3,(xy^2xy^3)^{-1}\}$. When showing this, we make use of the equation $(xy)^{-1}z=x^{p-1}y^{p-1}$ and $(xy^2xy^3)^{-1}=y^{p-5}x^{p-2}z^{2}$. We get the equations,
\begin{minipage}[b]{0.4\linewidth} 
\begin{align*}x^{-1}xy^{2}x &= y^2x; \\ y^{-1}xy^{2}y &= yx^2z^2;\\ y^{-1}x^{-1}xy^2xy &= y^2xz;\\ \end{align*} 
\end{minipage}
\begin{minipage}[b]{0.4\linewidth} 
\begin{align*}x^{-1}xy^{3}x &= y^3x;\\ y^{-1}xy^{3}y &= y^2x^2z^3;\\ y^{-1}x^{-1}xy^3xy &= y^3xz;\\ \end{align*} 
\end{minipage}
\begin{align*}x^{-1}y^{p-5}x^{p-2}z^2x &=y^{p-5}x^{2p-4}z^{2+(p-5)(p-1)} ;\\ y^{-1}y^{p-5}x^{p-2}z^2y &=y^{p-5}x^{p-2}z^{p} ;\\ y^{-1}x^{-1}y^{p-5}x^{p-2}z^2xy &=y^{p-5}x^{2p-2}z^{2p-1} .\\ \end{align*}

Therefore, as centralizing does not occur for $p\ge 7$, the result follows.
\end{proof}

\begin{remark}
The group given by Lemma \ref{NathanLem} for $p=7$ is the second group in a family of groups in \cite[Theorem 3.2]{BarkerOther}. There, it arises as a $7$-quotient of a finite index subgroup of an infinite group with special presentation related to a finite projective plane.
\end{remark}

For groups of order $5^3$, a MAGMA search reveals that the only Beauville $5$-group of order $5^3$ is the one given by
\[G:=\langle x,y,z|x^5,y^5,z^5,[x,y]=z\rangle,\]
with Beauville structure $\{\{x,y\},\{xy^2xy^4\}\}$.

 The above has the following consequence. 

\begin{cor}
The smallest non-abelian Beauville $p$-group for $p\ge5$ has order $p^3$.
\end{cor}


\section{Groups of order $p^4$}\label{OrderP4}
The classification of groups of order $p^4$ is well-known; this result is due to \cite{Holder}. We list the non-abelian 2-generated groups of order $p^4$ in Table $1$ for $p$ odd and Table $2$ for $p=2$. The only abelian Beauville group of order $p^4$ is $C_{p^2}\times C_{p^2}$ for $p>3$.

\begin{table}
\begin{center}
\begin{tabular}{|l|l|c|}
\hline
Name&Presentation&Beauville?\\
\hline
$G_1$&$\langle x,y|x^{p^3},y^p,x^y=x^{1+p^2}\rangle$&No\\
$G_2$&$\langle x,y|x^{p^2},y^{p^2},x^y=x^{p+1}\rangle$&Yes $(p>3)$\\
$G_3$&$\langle x,y,z|x^{p^2},y^p,z^p,[x,z]=y\rangle$&No\\
$G_4$&$\langle x,y,z|x^{p^2},y^p,z^p,x^y=x^{p+1},[x,z]=y\rangle$&No\\
$G_5$&$\langle x,y,z|x^{p^2},y^p,z^p=x^p,x^y=x^{p+1},[x,z]=y\rangle$&No\\
$G_6$&$\langle x,y,z|x^{p^2},y^p,z^p=x^{p\alpha},x^y=x^{p+1},[x,z]=y\rangle$&No\\
$G_7$ ($p>3$)&$\langle w,x,y,z|w^p,x^p,y^p,z^p,[y,z]=x,[x,z]=w\rangle$&Yes $(p>3)$\\
$G_8$ ($p=3$)&$\langle x,y,z|x^9,y^3,z^3,[x,z]=y,[y,z]=x^3\rangle$&No\\
\hline
\end{tabular}

\label{pfour}\caption{The non-abelian 2-generated groups of order $p^4$, $p$ odd. In the groups $G_3,...,G_6$ and $G_8$, the presence of the relation $[x,z]=y$ shows that the group is 2-generated. In $G_7$ the presence of the relations $[y,z]=x$ and $[x,z]=w$ show that the group is 2-generated. In $G_6$ $\alpha$ is any quadratic non-residue (mod $p$).}
\end{center}
\end{table}

The group $G_1$ is not Beauville as a special case of Lemma \ref{firstgen}. The groups $G_3$, $G_4$ $G_5$, $G_6$ and $G_8$ are never Beauville groups by an argument analogous to the proof of Lemma \ref{firstgen}, that is, in each case all elements of order $p$ are contained in a proper subgroup, so any generating set must contain an element of order $p^2$, but since all elements of order $p^2$ power up to the same elements of order $p$, we cannot have a Beauville structure. The groups in Table 2 are easily checked by computer not to be Beauville groups.

The group $G_{2}$ is a Beauville group for $p>3$ by Lemma \ref{BenLemma} and $G_7$ is a Beauville group for $p>3$ by an argument analogous to the proof of Lemma \ref{NathanLem} showing that $\{\{w,z\},\{wz^2,wz^3\}\}$ is a Beauville structure. We can state the above information in the following lemma. 

\begin{lemma}
For any prime $p\ge 5$ the groups $G_2$ and $G_7$ are non-abelian Beauville groups of order $p^4$.

For $p=3$, the groups $G_2$ and $G_7$ are not Beauville groups.
\end{lemma}

\begin{table}
\begin{center}
\begin{tabular}{|l|l|}
\hline
Name&Presentation\\
\hline
$G_1,G_2,G_3$& as in Table 1 \\
$G^{'}_4$&$\langle x,y|x^8,y^2,x^y=x^7\rangle$\\
$G^{'}_5$&$\langle x,y|x^8,y^2,x^y=x^3\rangle$\\
$G^{'}_6$&$\langle x,y|x^8,y^4,x^y=x^{-1},x^4=y^2\rangle$\\
\hline
\end{tabular}
\label{pfoureven}\caption{The non-abelian 2-generated groups of order $2^4$.}
\end{center}
\end{table}


\section{Groups of order $p^5$}\label{p5}

Computer calculations using MAGMA show that this is the first occurrence of a Beauville $3$-group. This group is, in fact, the only Beauville group of order $3^5$.

\begin{example}\label{p3}
The group
\[\langle x,y,z,w,t|x^3,y^3,z^3,w^3,t^3,y^x=yz,z^x=zw,z^y=zt\rangle\]
is a non-abelian Beauville group of order $3^5$ with Beauville structure given by $\{\{x,y\},\{xt,y^2w\}\}$.
\end{example}

\begin{table}
\begin{center}
\begin{tabular}{|l|l|l|c|}
\hline
$p$&$n$&$h(p)$&$g(p)$\\
\hline
2&-&19&0\\
3&3&29 &1 \\
5&2, 3, 7, 8, 9, 10, 11, 12, 13, 14, 19, 20, 23, 30, 33&37 &15\\
7&2, 3, 7, 8, 9, 10, 11, 12, 13, 14, 15, 16, 21, 22, 25,& & \\
 & 32, 37& 41&17\\
11&2, 3, 7, 8, 9, 10, 11, 12, 13, 14, 15, 16, 17, 18, 19, & &\\
&20, 25, 26, 29, 36, 39&41&21\\
13 &2, 3, 7, 8, 9, 10, 11, 12, 13, 14, 15, 16, 17, 18, 19, &&\\
 &20, 21, 22, 27, 28, 31,38, 43&49& 23 \\
17 & 2, 3, 7, 8, 9, 10, 11, 12, 13, 14, 15, 16, 17, 18, 19, &&\\ 
  &20, 21, 22, 23, 24, 25,26, 31,32, 35, 42, 45&49& 27\\
 19 &2, 3, 7, 8, 9, 10, 11, 12, 13, 14, 15, 16, 17, 18, 19,&&\\
&20, 21, 22, 23, 24, 25, 26, 27, 28, 33, 34, 37, 44, 49&53&29\\
\hline
\end{tabular}
\label{pfive}\caption{The groups {\ttfamily SmallGroup}($p^5$,$n$) for $p\leq19$ a prime that have Beauville structures. $h(p)$ (respectively $g(p)$) is the number of $2$-generated (respectively Beauville) groups of order $p^5$.}
\end{center}
\end{table}

The computer program MAGMA was further used to explore the possible Beauville groups of order $p^5$, for $p>3$. The results of our computer experimentations are presented in Table $3$. We note that there are no abelian Beauville groups of order $p^5$.

We observed that for each prime $5\le p\le 19$ there are exactly $p+10$ Beauville groups of order $p^5$. The presentations for the $p+10$ groups are given below, seven $H_{i}$ groups and $p+3$ $H_{i,j,k,l}$ groups. The remainder of this section is devoted to proving Theorem \ref{familyP5}. We start by showing that five of the seven $H_{i}$ groups are Beauville groups. We follow this up by using the work of \cite[Section 4.5, part (6)]{J} to analyze a family of non-isomorphic groups given by the groups $H_{i,j,k,l}$.

Let ${\bf X}=$ $\{x,y,z,w,t\}$ and set $H_{i}:=\langle {\bf X}| {\bf R}_{i}\rangle$ for the below relations,
\[{\bf R}_{1}=\{x^p=w,y^p=t,z^p,w^p,t^p,[y,x]=z\},\]
\[{\bf R}_{2}=\{x^p,y^p,z^p,w^p,t^p,[y,x]=z,[z,x]=w,[z,y]=t\},\]
\[{\bf R}_{3}=\{x^p=w,y^p=t,z^p,w^p,t^p,[y,x]=z,[z,x]=t\},\]
\[{\bf R}_{4}=\{x^p=w,y^p=t^r,z^p,w^p,t^p,[y,x]=z,[z,x]=t\},\]
where $r$ is taken as $2,5,6,7,6,10$ for $p=5,7,11,13,17,19$ and
\[{\bf R}_{5}=\{x^p=w,y^p=t,z^p,w^p,t^p,[y,x]=z,[z,x]=t,[z,y]=t\},\]
\[{\bf R}_{6}=\{x^p,y^p,z^p,w^p,t^p,[y,x]=z,[z,x]=w,[w,x]=t\}\]
\[{\bf R}_{7}=\{x^p,y^p,z^p,w^p,t^p,[y,x]=z,[z,x]=w,[z,y]=t,[w,x]=t\}.\]
\begin{remark}
It would be interesting to know how $r$, which appears in the set of relations ${\bf R}_4$, varies as a function of $p$.
\end{remark}

The above $H_{i}$ groups correspond to Beauville groups for $5\leq p \leq 19$. We now look to \cite[Section 4]{FJ} on lifting Beauville structures to extend the computational results to primes $p>19$.

\begin{definition}\label{faith}
Let $G$ be a finite group with a normal subgroup $N$. An element $g$ of $G$ is faithfully represented in $G/N$ if $\langle g\rangle \cap N =\{e\}$. 
\end{definition}

If $T=\{g_{1},...,g_{k}\}$ is a $k$-tuple of elements of $G$, we say that this $k$-tuple is faithfully represented in $G/N$ if $\langle g_{i}\rangle \cap N=\{e\}$ for $1\leq i \leq k$.

\begin{lemma}\cite[Lemma 4.2]{FJ}\label{lift}
Let $G$ have generating triples $\{x_{i},y_{i},z_{i}\}$ with $x_{i}y_{i}z_{i}=e$ for $i=1,2$ and a normal subgroup $N$ such that at least one of these triples is faithfully represented in $G/N$.
 
 If the images of these triples corresponds to a Beauville structure for $G/N$, then these triples correspond to a Beauville structure for $G$. 
\end{lemma}

We can now make the following conclusions for some of the group structures $H_{i}=\langle {\bf X}| {\bf R}_{i}\rangle$.

\begin{lemma}\label{lem1}
Let $H_{i}=\langle {\bf X}| {\bf R}_{i}\rangle$ for $i=2,6,7$ and $p\ge 5$ a prime. Then, $H_{i}$ is a Beauville group of order $p^5$.
\end{lemma}

\begin{proof}
Firstly, for $p=5$ MAGMA calculations show that the groups $H_{i}$ for $i=2,6,7$ have Beauville structures corresponding to $\{\{x,y\},\{xy^2,xy^4\}\}$.

Secondly, let $p\ge 7$. For each group $H_{i}$ the center $Z_{i}=Z(H_{i})$ is given by the subgroup $\langle t,w\rangle$ and $\{x,y\},\{xy^2,xy^3\}$ are two generating sets for the groups $H_{i}$ for $i=2,6,7$. The quotient group $H_{i}/Z_{i}$ is isomorphic to the group $G$ given in Lemma \ref{NathanLem}. Clearly, the images of $x,y$ and $xy$ in $H_{i}/Z_{i}$ are faithfully represented (in the sense of Definition \ref{faith}) and correspond with the Beauville structure $\{\{x,y\},\{xy^2,xy^3\}\}$ for the group $G$.

Thus, by Lemma \ref{lift} we see that the Beauville structure $\{\{x,y\},\{xy^2,xy^3\}\}$ lifts to a Beauville structure for the groups $H_{i}$ for $i=2,6,7$.
\end{proof}

\begin{lemma}\label{lem2}
Let $H_{1}=\langle {\bf X} | {\bf R}_{1}\rangle$ and $p\ge 5$ a prime. Then, $H_{1}$ is a Beauville group of order $p^5$.
\end{lemma}

\begin{proof}
By Lemma \ref{Oddpowers}, with $n=2$, we see that the groups $H_{1}$ have Beauville structures corresponding to $\{\{x,y\},\{xy^2,xy^4\}\}$.
\end{proof}

\begin{lemma}\label{lem3}
Let $H_{5}=\langle {\bf X} | {\bf R}_{5}\rangle$ and $p\ge 5$ a prime. Then, $H_{5}$ is a Beauville group of order $p^5$.
\end{lemma}

\begin{proof}
We claim that the groups $H_{5}$ for $p\ge 5$ have Beauville structures corresponding to $\{\{x,y\},\{xy^2,xy^4\}\}$.

It is clear that $\{x,y\}$ and $\{xy^2,xy^4\}$ are generating sets for $H_{5}$. Now, given that $x^p=w$, $y^p=t$, $[x,y]=z$, $[z,x]=[z,y]=t$ and the center $Z(H_{5})=\langle w,t \rangle$ we see that  
\[\Sigma(x,y)=\]
\[\{e\}\bigcup\left(\bigcup_{i=1}^{p^2-1}\{x^{i},y^{i},x^{i}y^{i}\}\langle z\rangle\langle y^p\rangle\right)\setminus \bigcup_{i,j,k=1}^{p-1}\{x^{ip}y^{jp}z^k,y^{ip}y^{jp}z^k,x^{ip}y^{ip}y^{jp}z^k\},\]
and
\[\Sigma(xy^2,xy^4)=\{e\}\bigcup\left(\bigcup_{i=1}^{p^2-1}\{x^{i}y^{2i},x^{i}y^{4i},x^{2i}y^{6i}\}\langle z\rangle\langle y^p\rangle\right)\] \[\setminus \bigcup_{i,j,k=1}^{p-1}\{x^{ip}y^{2ip}y^{jp}z^k,x^{ip}y^{4ip}y^{jp}z^k,x^{2ip}y^{6ip}y^{jp}z^k\}.\]
We prefer to write $w$ in terms of $x^{ip}$ and $t$ in terms of $y^{ip}$ for $0\leq i\leq p-1$. Therefore, $\Sigma(x,y)\cap\Sigma(xy^2,xy^4)=\{e\}$.
\end{proof}

We are now left with the groups given by relations ${\bf R}_{i}$ for $i=3,4$. We cannot lift Beauville structures from groups of order $< p^5$ to the groups $H_{i}$ for $i=3,4$ as any normal subgroup $N_{i}$ of $H_{i}$ would decrease the order of the generators $x,y$. Thus, $x,y$ would not be faithfully represented in $H_{i}/N_{i}$.

We now have the following groups for selected values of $i,j,k,l\in \{0,...,p-1\}$. We find $p+3$ non-isomorphic groups for $5\leq p \leq 19$ give rise to Beauville $p$-groups with the following presentations,
\[H_{i,j,k,l}:=\langle x,y,z,w,t|x^p=w^it^j,y^p=w^kt^l,z^p,w^p,t^p,[x,y]=z,[x,z]=w,[y,z]=t\rangle.\]
These groups correspond to the groups {\textnormal{\ttfamily SmallGroup}}($p^5$, n) for $7\leq n\leq p+9$, as given by the MAGMA (and GAP) small groups library.

From \cite[Section 4.5, part (6)]{J}, the group structures for $p$-groups of order $p^5$ for $p>3$ are listed. The groups having the structure of the groups $H_{i,j,k,l}$ are thus given in the classification. We will use this classification to find Beauville structures for the groups $H_{i,j,k,l}$ to extend the computational results to primes $p>19$. 

We can state the following lemma, which is a consequence of the classification of groups of order $p^5$.

\begin{lemma}\label{lemHijkl}
If $p>3$ a prime, then there are $p+7$ non-isomorphic groups of the following form,
\[H_{i,j,k,l}:=\langle x,y,z,w,t|x^p=w^it^j,y^p=w^kt^l,z^p,w^p,t^p,[x,y]=z,[x,z]=w,[y,z]=t\rangle\]
where $i,j,k,l\in \{0,...,p-1\}$.
\end{lemma}

\begin{proof}
From \cite[Section 4.5, part (6)]{J}, we see that there are 
\[1+\frac{1}{2}(p-1)+2+1+\frac{1}{2}(p-1)+1+2+1=p+7\]
groups of this form.
\end{proof}

We are now in a position to prove Theorem \ref{familyP5}, which was stated in the Introduction. It is convenient to note that all the groups $H_{i,j,k,l}$ have center $Z_{i,j,k,l}=\langle w,t\rangle$ and $H_{i,j,k,l}/Z_{i,j,k,l}\cong G$, the group given by Lemma \ref{NathanLem}. 

\begin{proofofT}{Theorem \ref{familyP5}} 
Firstly, by Lemmas \ref{lem1}, \ref{lem2} and \ref{lem3} we have five Beauville groups for each prime $p>3$.

Secondly, we consider the $p+7$ non-isomorphic groups $H_{i,j,k,l}$ given by Lemma \ref{lemHijkl}. We note that the group given by $H_{0,0,0,0}$ corresponds to $H_{2}$ and thus (as we do not want to count the group twice) we have $p+6$ non-isomorphic groups of the form $H_{i,j,k,l}$ to account for. 

The groups corresponding to $\Phi_{6}(2111)b_{r}$ in \cite[Section 4.5, part (6)]{J} cannot admit a Beauville structure as $x^p=e$, $y^p=w^r$ where $r=1$ or $\nu$ (the smallest positive integer which is a non-quadratic residue modulo $p$) \emph{i.e.} the groups $H_{0,0,r,0}$. Similarly, the group given by $\Phi_{6}(2111)a$ in \cite[Section 4.5, part (6)]{J} cannot admit a Beauville structure as $x^p=w$, $y^p=e$, \emph{i.e.} the group $H_{1,0,0,0}$. We are therefore left with $p+3$ non-isomorphic groups to analyze.

The remaining $p+3$ groups $H_{i,j,k,l}$ have nontivial words $u(w,t), v(u,t)$ such that $x^p=u(w,t)$ and $y^p=v(w,t)$. As the words $u,v$ are made up of elements of the center $Z_{i,j,k,l}$ of the groups $H_{i,j,k,l}$ and the order of the elements $x,y$ is $p^2$, we see that the remaining $p+3$ groups satisfy the criteria $\Sigma(x,y)\cap \Sigma(xy^2,xy^4)=\{e\}$ for $p>3$. That is, each element of the form $x^ay^bz^c$ (with both $a\ne0$ and $b\ne 0$) is conjugate to elements of the form $x^ay^bz^ds(w,t)$, where $s(w,t)$ is a word in $w,t$. Therefore, $\{\{x,y\},\{xy^2,xy^4\}\}$ is a Beauville structure for the remaining $p+3$ groups. The result then follows.
\end{proofofT}

We see for $5\leq p\leq 19$ that the number of groups found to have Beauville structures is $p+10$. From the above work, we are led to make the following conjecture.

\begin{conjecture}
For all $p\ge 5$, the number of Beauville $p$-groups of order $p^{5}$  is given by $g(p)=p+10$.

In particular, $H_3$ and $H_4$ are Beauville groups for $p\ge 5$.
\end{conjecture}

In the preceding paragraphs we produced $p+8$ groups of order $p^5$ that admit a Beauville structure. 

For groups of order $p^5$, the number of  $2$-generated groups is approximately half of the total number of groups. We see from \cite{J}, that the exact number of $2$-generated $p$-groups of order $p^5$ for $p\ge 5$ is given by
\[h(p)=p+26+2\gcd(p-1,3)+\gcd(p-1,4).\]

Thus, $h(p)\sim p$ as $p\to \infty$. The function $h(p)$ is an obvious upper bound for the number of Beauville groups of order $p^5$. Since $p+36\ge h(p)> g(p) \ge p+8$ we get that $g(p)\sim p$ as $p\to \infty$ and so,
\[\lim_{p\to \infty}\frac{g(p)}{h(p)}=1.\]
Thus, the proportion of $2$-generated groups of order $p^5$ which are Beauville tends to $1$ as $p$ tends to infinity, which establishes Corollary \ref{proportionP5}.

\section{Remarks on Groups of order $p^6$}

For groups of order $p^6$, we used MAGMA to determine that there are no Beauville $2$-groups and only three Beauville $3$-groups. These groups correspond to the groups {\textnormal{\ttfamily SmallGroup}}($3^6$, n) for $n=34,37,40$.

\begin{remark}
It is interesting to note that Corollary \ref{proportionP6} also holds for non-abelian $2$-generated groups of order $p^6$ since there are only $3$ abelian ones.
\end{remark}

For $p>3$, we would like an asymptotic result for groups of order $p^6$, similar to that in Section \ref{p5} for $p^5$. Using \cite[Theorem 2 and Table 1]{OBrien}, we see that there are in total 
\[f(p)=10p + 62 + 14\gcd(3,p-1) + 7\gcd(4,p-1) + 2\gcd(5,p-1)\]
 $2$-generated groups of order $p^6$ for $p>3$ a prime. Thus, $f(p)\sim 10p$ as $p\to \infty$.

From \cite[Theorem 2]{OBrien}, the family of groups of order $p^6$ given by \textquotedblleft{\em$3) \ \langle a, b| b^p, \ \text{class 2}\rangle$}" give rise to $p+15$ non-isomorphic groups (see \cite[Table 1]{OBrien}). One can generate these group presentations for each $p$ a prime by the following MAGMA code:
\begin{verbatim}
> G:=Group<a,b|b^p>;
> P:=pQuotient(G,p,2);
> D:=Descendants(P: OrderBound := p^6);
> D := [d: d in D | #d eq p^6];
\end{verbatim}

Each of the groups contained in $D$ is $2$-generated, say by $x$ and $y$. We find that, for each $p$ a prime, there exists a family of non-isomorphic groups contained in $D$ given by the following presentations,

\[K_{r}=\langle x,y,z,u,v,w| x^p=u,y^p=w^r,z^p,u^p=v,v^p,w^p,[y,x]=z,[z,x]=v,[z,y]=w\rangle,\]
for $r=1,...,p-1$. 

It follows that all of the $p-1$ groups have $o(x)\ne o(y)$. You can clearly see, given the above group structures, if $o(x)\ne o(y)$ then $K_{r}$ does not have a Beauville structure (this is similar to the third paragraph of the proof of Theorem \ref{familyP5}, Section \ref{p5}). That is, any second set of generators one tries to construct will have elements of the form $x^ay^b$ and so if $o(x)\ne o(y)$ we will have $\Sigma(x,y)\cap\Sigma(x^ay^b,x^cy^d)\ne\{e\}$. Therefore, we obtain a family of $p-1$ $2$-generated non-Beauville groups of order $p^6$, which proves Theorem \ref{familyP6} and establishes Corollary \ref{proportionP6}.

\section{Acknowledgments}

We wish to thank the organizers and participants of the Geometric Presentations of Finite and Infinite Groups conference held in the University of Birmingham in 2011, without whose assistance and useful conversations this work would not be possible. We also thank Christopher Voll for a helpful correspondence and Derek Holt for very helpful conversations.

We thank the referee for careful reading of the paper and the useful comments and suggestions made.

\bibliographystyle{alpha}

\end{document}